\documentclass[12pt]{amsart}
\usepackage[latin1]{inputenc}
\usepackage[english]{babel}
\usepackage{amsmath}
\usepackage{amsthm}
\usepackage{amssymb}
\usepackage{amsmath, amsthm, amsfonts, mathrsfs, amsfonts}
\usepackage{latexsym}
\usepackage{textcomp}
\usepackage{enumerate}
\usepackage{ulem}

\usepackage{xcolor}

\theoremstyle{definition}
\newtheorem{definition}{Definition}[section]

\theoremstyle{remark}
\newtheorem{remark}[definition]{Remark}

\theoremstyle{plain}

\theoremstyle{plain}
\newtheorem{cor}[definition]{Corollary}

\theoremstyle{plain}
\newtheorem{conj}{Conjecture}

\theoremstyle{plain}
\newtheorem{lemma}[definition]{Lemma}

\theoremstyle{plain}
\newtheorem{theorem}[definition]{Theorem}

\newtheorem{theoremA}[]{Theorem}

\theoremstyle{remark}

\theoremstyle{remark}
\newtheorem{notation}[definition]{Notation}

\theoremstyle{definition}

\newcommand{\Syl}{\mathrm{Syl}}

\newcommand{\Sym}{\mathrm{Sym}}

\newcommand{\norm}{\mathrel{\unlhd}}

\def \Z {\mathrm {Z}}

\def \Syl {\hbox {\rm Syl}}

\title[Nilpotency and non-commuting graphs]{A conjecture related to the nilpotency of groups with isomorphic non-commuting graphs}

\author[V. Grazian]{Valentina Grazian}\address[V. Grazian]{
Department of Mathematics and Applications\\
University of Milano -- Bicocca\\
Via Roberto Cozzi 55, 20125 Milano \\ Italy} \email{valentina.grazian@unimib.it}

\author[C. Monetta]{Carmine Monetta}\address[C. Monetta]{Department of Mathematics \\ University of Salerno \\
via Giovanni Paolo II 132, 84084 Fisciano (SA)\\ Italy}
\email[Corresponding author]{cmonetta@unisa.it}

\begin{document}

\maketitle

\begin{abstract} In this work we discuss whether the non-commuting graph of a finite group can determine its nilpotency. More precisely, Abdollahi, Akbari and Maimani conjectured that if $G$ and $H$ are finite groups with isomorphic non-commuting graphs and $G$ is nilpotent, then $H$ must be nilpotent as well (Conjecture \ref{conjNilp}). 
We characterize the structure of such an $H$ when $G$ is a finite AC-group, that is, a finite group in which all centralizers of non-central elements are abelian. As an application, we prove Conjecture \ref{conjNilp} for finite AC-groups whenever $|Z(G)|\geq|Z(H)|$.
\end{abstract}

\section{Introduction}
Given a finite group $G$, one can consider the graph whose vertices are (some) elements of $G$ and whose edges reflect a certain structure property of $G$: we refer to \cite{Cameron} for a survey on the topic. Such a technique has proven to be a valuable tool to study various aspects of finite groups. 
In this work we will focus our attention on the \emph{non-commuting} graph of the non-abelian finite group $G$, that is, the graph $\Gamma_G$ whose vertices are the non-central elements of $G$ and in which two vertices are joined by an edge if they do not commute. Notice that if $G$ and $H$ are finite groups with isomorphic non-commuting graphs $\Gamma_G \cong \Gamma_H$, then $|G| - |Z(G)| = |H| - |Z(H)|$. 
In \cite{AAM} the authors studied what structural properties can be deduced from the assumption that two finite groups have isomorphic non-commuting graphs, posing (among others) the following two conjectures:

\begin{conj}\cite[Conjecture 1.1]{AAM}\label{conjOrder} Let $G$ and $H$ be two non-abelian finite groups such that $\Gamma_G \cong \Gamma_H$.
Then $|G| = |H|$.
\end{conj}

\begin{conj}\cite[Conjecture 3.25]{AAM}\label{conjNilp} Let $G$ be a finite non-abelian nilpotent group and $H$ be a group such that $\Gamma_G \cong \Gamma_H$. Then $H$ is nilpotent.
\end{conj}

Conjecture \ref{conjNilp} appears also in the Kourovka Notebook \cite{KM}, stated as Problem 16.1. 
 In \cite[Theorem 3.24]{AAM} the authors proved that Conjecture \ref{conjOrder} implies Conjecture \ref{conjNilp}:

 \begin{theorem}\label{same.order}
Let $G$ be a finite non-abelian nilpotent group and $H$ be a group such that $\Gamma_G \cong \Gamma_H$. If $|G| = |H|$ then $H$ is nilpotent.
 \end{theorem}
 
Conjecture \ref{conjOrder} 
has been proven to be true when $G$ is a $p$-group (\cite[Theorem 1.2]{non.comm.p.groups}), a dihedral group (\cite[Proposition 3.19]{AAM}) or a simple group (\cite[Theorem A]{dar}), where in the latter case we can even conclude that $G \cong H$ (\cite[Theorem 1.1]{SW}). Moreover, in \cite[Proposition 3.3]{AAM} the authors showed that it is sufficient to prove Conjecture \ref{conjOrder}  for a specific class of finite groups: the class of finite groups in which all centralizers of non-central elements are abelian, called AC-groups. 
 However, in \cite{counter.non.comm}, the author constructed two AC-groups $G = P \times A$ and $H = Q \times B$ with isomorphic non-commuting graphs, where $P$ is a non-abelian $2$-group of order $2^{10}$, $Q$ is a non-abelian $5$-group of order $5^6$, $A$ and $B$ are abelian groups satisfying $2^5\cdot|A|=2^2 \cdot 5^3 \cdot |B|$, and in general, $|G| = 2^{10}|A| \neq 5^6|B| = |H|$. 
 Nevertheless, Conjecture \ref{conjNilp} is not affected by this example and is in fact still open. One of the major progresses made toward its proof is the following result:
 
\begin{theorem}\cite[Theorem 2.4]{non-comm.nilp}\label{nilp.two.sylow}
Let $G$ be a finite non-abelian nilpotent group and suppose $\Gamma_G \cong \Gamma_H$ for a group $H$. If $G$ has at least two distinct non-abelian Sylow subgroups and $|Z(G)| \geq |Z(H)|$ then $|G|=|H|$, (and so $H$ is nilpotent by Theorem~\ref{same.order}).
\end{theorem}

Theorem \ref{nilp.two.sylow} suggests that, in order to prove Conjecture \ref{conjNilp}, it is necessary to study the class of of finite non-abelian nilpotent groups having a unique non-abelian Sylow subgroup, that is, finite groups of the form $G=P \times A$, where $p$ is a prime, $P \in \Syl_p(G)$ is non-abelian and $A$ is an abelian $p'$-group. 
Taking inspiration from these considerations, we pose the following conjecture:

\begin{conj}
\label{conj1}
Let $p$ be a prime and suppose $G=P \times A$ is a finite group where $P \in \Syl_p(G)$ is non-abelian and $A$ is an abelian $p'$-group. If $\Gamma_G \cong \Gamma_H$ for a finite group $H$ and $|Z(G)| \geq |Z(H)|$ then $H=Q \times B$, where $q$ is a prime, $Q \in \Syl_q(H)$ is non-abelian and $B$ is an abelian $q'$-group. In particular, $H$ is nilpotent.
\end{conj}

If Conjecture \ref{conj1} is true, then, combining it with Theorem \ref{nilp.two.sylow}, we can immediately deduce Conjecture \ref{conjNilp} as long as $|Z(G)| \geq |Z(H)|$:

\begin{theorem}\label{3implies2}
Let $G$ be a finite non-abelian nilpotent group and suppose $\Gamma_G \cong \Gamma_H$ for a group $H$. If Conjecture \ref{conj1} holds and $|Z(G)| \geq |Z(H)|$ then $H$ is nilpotent.
\end{theorem}

As for Conjecture \ref{conj1}, it is known to be true if $G$ is a $p$-group (by \cite[Theorem 1.2]{non.comm.p.groups}) and if the graph $\Gamma_G$ is regular (as this means that the conjugacy classes of non-central elements of $G$ have only one size and we conclude by \cite[Theorem 1]{ito}).

In this work, we prove Conjecture \ref{conj1} when $G$ is an AC-group. More precisely, we obtain the following stronger result:

\begin{theoremA}\label{main}
Let $p$ be a prime and let $G = P \times A$ be a finite non-abelian nilpotent AC-group, where $P \in \Syl_p(G)$ is non-abelian and $A$ is an abelian $p'$-group. Let $H$ be a group such that $\Gamma_G \cong \Gamma_H$. Then $H$ is a finite AC-group and either 
\begin{enumerate}
\item $H = Q \times B$, where $q$ is a prime, $Q \in \Syl_q(H)$ is non-abelian and $B$ is an abelian $q'$-group; or
\item $|Z(H)| > |Z(G)|$, $[P \colon Z(P)] >p^4$, none of the maximal subgroups of $G$ is abelian (so $P$ does not have maximal nilpotecncy class) and $H/Z(H)$ is a Frobenius group with Frobenius kernel and complement $F/Z(H)$ and $K/Z(H)$, respectively, where $K$ is an abelian subgroup of $H$, $Z(F) = Z(H)$, and 
$[ F \colon Z(H)] = q^t$ for some prime $q \neq p$. 
\end{enumerate}
\end{theoremA}

Given that all finite nilpotent AC-groups are of the form described in Conjecture \ref{conj1} (see Theorem \ref{classification}), we can now state our contribution to Conjecture \ref{conjNilp} as a direct consequence of Theorem \ref{main}:
\begin{cor}
Let $G$ be a finite non-abelian nilpotent AC-group and $H$ be a group such that $\Gamma_G \cong \Gamma_H$. If $|Z(G)| \geq |Z(H)|$ then $H$ is nilpotent.
\end{cor}


\section{Finite AC-groups}
In this section we describe some properties of AC-groups, starting by recalling their definition:

\begin{definition}
A finite group $G$ is said to be an AC-group if all centralizers in $G$ of non-central elements of $G$ are abelian. 
\end{definition}

The structure of non-abelian AC-groups was deeply investigated by Schmidt (\cite{Schmidt1970}) and Rocke (\cite{RockePhD} and \cite{Rocke}). We report some useful properties of such groups.

First of all note that the centralizers of non-central elements of a finite AC-group $G$ are exactly the maximal abelian subgroups of $G$.

\begin{lemma}\cite[Lemma 3.2]{Rocke}\label{first.prop}  The following are equivalent for a finite non-abelian group~$G$.
\begin{enumerate}
\item $G$ is an AC-group.
\item If $[x,y] = 1$, then $C_G(x) = C_G(y)$, whenever $x,y \in G\backslash Z(G)$.
\item If $[x,y] = [x, z] = 1$, then $[y, z] = 1$, whenever $x \in G\backslash Z(G)$.
\item If $A$ and $B$ are subgroups of $G$ and $Z(G) < C_G(A) \leq C_G(B) < G$, then $C_G(A) = C_G(B)$.
\end{enumerate}
\end{lemma}

\begin{cor}\label{centralizers.intersection}
Let $G$ be a finite non-abelian AC-group and let $x,y\in G \backslash Z(G)$. Then the following are equivalent
\begin{enumerate}
\item $[x,y]\neq 1$; 
\item $C_G(x) \cap C_G(y) = Z(G)$;
\item $C_G(x) \neq C_G(y)$.
\end{enumerate} 
\end{cor}

\begin{proof} 
Suppose $[x,y]\neq 1$. Note that $C_G(x) \cap C_G(y) = C_G(\langle x,y\rangle)$. By Lemma \ref{first.prop}(4) if $Z(G) < C_G(\langle x,y\rangle)$ then $C_G(\langle x \rangle )=C_G(\langle x,y\rangle) = C_G(\langle y \rangle)$. Thus $y \in C_G(x)$, a contradiction. Hence we must have $Z(G) = C_G(\langle x,y\rangle) = C_G(x) \cap C_G(y)$, proving that (1) implies (2).
Clearly (2) implies (3), since $Z(G) < Z(G)\langle x\rangle \leq C_G(x)$. Finally, (3) implies (1) by Lemma \ref{first.prop}(2).
\end{proof}

It is clear that the property of being an AC-group is preserved by subgroups:

\begin{lemma}\cite[Remark 3.3(c)]{Rocke}
If $G$ is a finite AC-group and $H\leq G$ then $H$ is a finite AC-group.
\end{lemma}

\begin{notation}
If $G$ is a finite non-abelian group, set $$\mathcal{C}(G) = \{C_G(x) \mid x \in G \backslash Z(G)\}.$$ 
\end{notation}

For a finite non-abelian AC-group $G$, we can compute the order of $\mathcal{C}(G)$:

\begin{lemma}\label{clique}
Suppose that $G$ is a finite non-abelian $AC$-group 
Then 
\[|\mathcal{C}(G)| =  -[G \colon Z(G)] + 1 + \sum_{C \in \mathcal{C}(G)}[C \colon Z(G)].\] 
In particular, if $G/Z(G)$ is a $p$-group then $|\mathcal{C}(G)|  \equiv 1 \mod p$.
\end{lemma}

\begin{proof}
By Corollary \ref{centralizers.intersection} we have $C_G(x) \cap C_G(y) = Z(G)$ whenever $C_G(x) \neq C_G(y)$. Therefore
\[ |G| = \sum_{C \in \mathcal{C}(G)}|C| - |\mathcal{C}(G)||Z(G)| + |Z(G)|\]
that, dividing all by $|Z(G)|$, gives the result.
\end{proof}

In \cite[Satz 5.12]{Schmidt1970} Schmidt classified finite non-abelian solvable AC-groups, also indicating the order of the set $\mathcal{C}(G)$: 

\begin{theorem}\label{classification} Let $G$ be a finite non-abelian solvable AC-group. Then $G$ satisfies one of the following properties:
\begin{enumerate}
    \item $G$ is non-nilpotent and it has an abelian normal subgroup $N$ of prime index; moreover $|\mathcal{C}(G)| = [N \colon Z(G)] + 1$;
 
 \item $G/Z(G)$ is a Frobenius group with Frobenius kernel and complement $F/Z(G)$ and $K/Z(G)$, respectively, and $F$ and $G$ are abelian subgroups of $G$; moreover $|\mathcal{C}(G)| = [F \colon Z(G)] + 1$;
 
 \item $G/Z(G)$ is a Frobenius group with Frobenius kernel and complement $F/Z(G)$ and
$K/Z(G)$, respectively, $K$ is an abelian subgroup of $G$, $Z(F) = Z(G)$, and $F/Z(F)$ is of prime power order; moreover $|\mathcal{C}(G)| = [F \colon Z(G)] + |\mathcal{C}(F)|$;

\item $G/Z(G) \cong \Sym(4)$ and $V$ is a non-abelian subgroup of $G$ such that $V/Z(G)$ is the Klein $4$-group of $G/Z(G)$; moreover $|\mathcal{C}(G)| = 13$;

\item $G = P \times A$ , where $P$ is an AC-subgroup of prime power order and $A$ is an abelian group. 

\end{enumerate}

\end{theorem}

\begin{remark}
Note that if $G$ is a finite nilpotent AC -group, then $G$ is of type (5) of Theorem \ref{classification}, that is, $G = P \times A$, where $P \in \Syl_p(G)$ for some prime $p$ and $A$ is an abelian subgroup of $G$, that we can assume to be  of $p'$-order. In particular, $G/Z(G) \cong P/Z(P)$ is a $p$-group, so from Lemma \ref{clique} we deduce that $|\mathcal{C}(G)| \equiv 1 \mod p$, and for all $x \in G\backslash Z(G)$ we have $C_G(x) = C_P(x) \times A$.
\end{remark}

We conclude this section with some properties of AC-groups of prime power order.

\begin{lemma}\label{lem:p-groups}
Let $P$ be a finite non-abelian AC-group and suppose that $P$ is a $p$-group for some prime $p$. Let $x \in Z_2(P) \backslash Z(P)$. Set $[P \colon Z(P)] = p^n$ and $[C_P(x) \colon Z(P)] = p^s$. If $c$ denotes the nilpotency class of $P$, then
\begin{enumerate}
\item $P' \leq C_P(x) \norm P$ (in particular $P'$ is abelian);
\item if $c > 2$ then $C_P(x)$ is the unique normal centralizer of $P$ and its order is maximum among the non-central element-centralizers of $G$ (in particular $Z_2(P) \leq C_P(x) = C_P(Z_2)$);
\item if $[P \colon C_P(x)] \geq p^2$ then $P/Z(P)$ has exponent $p$;
\item if $[P \colon C_P(x)] \geq p^2$ then $c \leq p$ and $Z_i(P) \leq C_P(x)$ for every $1 \leq i \leq c-1$;
\item if $y \in P \backslash C_P(x)$ and $[C_P(y) \colon Z(P)] = p^t$ then $n\geq s+t$ (in particular if $c > 2$ then $n \geq 2t$).
\end{enumerate}

\end{lemma}

\begin{proof}
Let $x \in Z_2(P) \backslash Z(P)$. Then $x$ commutes with $P'=[P,P]$ and so $P' \leq C_P(x)$, implying that $C_P(x) \norm P$ and proving (1).
By \cite[Lemma 3.8]{Rocke} if $c > 2$ then $C_P(x)$ is the unique normal centralizer of $P$ and by \cite[Proposition 2.7]{non.comm.p.groups} the order of $C_P(x)$ is maximal. To complete the proof of part (2), note that this, combined with part (1), implies $C_P(x)=C_P(y)$ for every $x,y \in Z_2(P) \backslash Z(P)$.
Part (3) is \cite[Theorem 3.13(b)]{Rocke}  and (4) is a combination of \cite[Lemma 3.14(c) and Lemma 3.15]{Rocke}.
Finally, since $C_P(x) \cap C_P(y) = Z(G)$ by Corollary \ref{centralizers.intersection} and $C_P(x) \norm P$ by part (1), we deduce that $C_P(y)/Z(P) \cong C_P(x)C_P(y)/C_P(x) \leq P/C_P(x)$. Hence $p^t \leq p^{n-s}$, giving $n \geq s +t$. In particular, if $c >2$ then $s \geq t$ by part (2) and so $n \geq 2t$.  
\end{proof}

\section{Non-commuting graphs of finite AC-groups}
In this section we analyse the properties of finite AC-groups having isomorphic non-commuting graphs.
We begin recalling the definition of non-commuting graph and some considerations on groups having isomorphic non-commuting graphs. 

\begin{definition}
If $G$ is a finite group, then the non-commutative graph of $G$, denoted by $\Gamma_G$, is the graph whose vertices are the elements of $G \backslash Z(G)$ and in which two vertices $x$ and $y$ are adjacent if and only if $xy \neq yx$.
\end{definition}

\begin{remark}
Suppose $G$ and $H$ are finite groups with a  graph-isomorphism $\phi \colon \Gamma_G \rightarrow~\Gamma_H$. First of all, the graphs  $\Gamma_G$ and $ \Gamma_H$ have the same number of vertices, so 

\begin{equation}\label{basic1} |G|-|Z(G)| = |H| - |Z(H)|. \end{equation}
For $x\in G \backslash Z(G)$, set $\Gamma_G(x) = \{y \in G \backslash Z(G) \mid y$ is adjacent to $x\} =G\backslash C_G(x)$. Then  $|\Gamma_G(x)| = |\Gamma_H(\phi(x))|$, implying
\begin{equation}\label{basic2}  |G|-|C_G(x)| = |H| - |C_H(\phi(x))|. \end{equation}
Combining Equations (\ref{basic1}) and (\ref{basic2}) we also obtain
\begin{equation}\label{basic3} |C_G(x)|-|Z(G)| = |C_H(\phi(x))| - |Z(H)|. \end{equation}
Equations (\ref{basic1}), (\ref{basic2}) and  (\ref{basic3}) will be used in the rest of this work without further reference.
\end{remark}

We now focus on the non-commuting graph of finite AC-groups. Note that the property of having abelian centralizers  can be easily read from the non-commuting graph of a finite group $G$. Indeed, $G$ is an AC-group if and only if $\Gamma_G$ is a complete multipartite graph (see Lemma \ref{first.prop}).
In particular, we immediately get the following (with the help of \cite[Lemma 3.1]{AAM} to prove that $H$ is finite):

\begin{lemma}\label{H.AC-group}
Let $G$ be a finite non-abelian AC-group and $H$ be a group such that $\Gamma_G \cong \Gamma_H$. Then $H$ is a finite AC-group. 
\end{lemma}

Recall that the clique number $w(\Gamma)$ of a graph $\Gamma$ is the largest possible number of vertices of a complete subgraph of $\Gamma$. The fact that distinct elements $x$ and $y$ of $G \backslash Z(G)$ are adjacent in $\Gamma_G$ if and only if $C_G(x) \neq C_G(y)$, proves the next result:

\begin{lemma} Let $G$ be a finite AC-group and $H$ be a group such that $\Gamma_G \cong \Gamma_H$. Then
\[ |\mathcal{C}(G)| = w(\Gamma_G)=w(\Gamma_H)=|\mathcal{C}(H)|.\]
\end{lemma}

From now on we focus on our main goal, that is understanding the structure of $H$ when $G$ is a nilpotent AC-group.

\begin{lemma}\label{solvable}
Let $G$ be a finite non-abelian nilpotent AC-group and $H$ be a group such that $\Gamma_G \cong \Gamma_H$. Then $H$ is a finite solvable AC-group. 
\end{lemma}

\begin{proof}
The fact that $H$ is a finite AC-group follows from Lemma \ref{H.AC-group}.
Aiming for a contradiction, suppose $H$ is non-solvable. Then \cite[Proposition 3.14]{AAM} gives that $|G|=|H|$. Thus Theorem \ref{same.order} implies that $H$ is nilpotent, a contradiction.
\end{proof}

Lemma \ref{solvable} tells us that we can suppose that $H$ is one of the group of Theorem \ref{classification}. This will be a crucial ingredient in the proof of Theorem \ref{main}.

\section{Proof of the main theorem}
In this section, let $G$ be a finite non-abelian nilpotent AC-group and $H$ be a group such that $\Gamma_G \cong \Gamma_H$. 

\begin{notation}
By Theorem \ref{classification} we can assume $G = P \times A$, where $P \in \Syl_p(G)$ for some prime $p$ and $A$ is an abelian subgroup of $G$ of $p'$-order. Set $p^n = [G \colon Z(G)] = [P \colon Z(P)]$, $p^r = |Z(P)|$ and  $a = |A|$, so $|Z(G)| = |Z(P)||A| = p^ra$ and $|G| = p^{n+r}a$. 
\end{notation}

By Lemma \ref{solvable}, the group $H$ is a finite solvable AC-group, and so it must correspond to one of the groups described in Theorem \ref{classification}. In particular, it is nilpotent if and only if it is of type (5). 
We will therefore analyze the various possibilities for $H$ as listed in Theorem \ref{classification}. As a first step, we show that if $H$ is not nilpotent then it must be of type (3). 

\begin{lemma}\label{type1}
$H$ is not of type (1).
\end{lemma}

\begin{proof}
 Let $q$ be a prime such that $[H \colon N] = q$.
Take $n\in N \backslash Z(H)$ and $h \in H \backslash N$. Since $H$ is a non-nilpotent $AC$-group, we get $N = C_H(n)$ and $C_H(h) = Z(H)\langle h \rangle$, with $|C_H(h)| = q|Z(H)|$.

Aiming for a contradiction, suppose that there exists an isomorphism of graphs $\Phi: \Gamma_H \rightarrow \Gamma_G$. Let $m = \phi(n)$ and $g = \phi(h)$ and set $M= C_G(m)$, $[M \colon Z(G)] = p^t$ and $[ C_G(g) \colon Z(G)] = p^u$. 
\bigskip

\begin{description}
\item[Claim 1] $p$ divides $[N \colon Z(H)]$.

\smallskip
\noindent
Proof: Note that $G/Z(G) \cong P/Z(P)$ is a $p$-group. Using Lemma \ref{clique} we get 
\[[N \colon \Z(H)] +1 = w(\Gamma_H) = w(\Gamma_G)  \equiv 1 \mod p\]
Hence $[N \colon \Z(H)]\equiv 0 \mod p$, that is, $p$ divides  $[N \colon Z(H)]$.

\bigskip
\item[Claim 2] $p^r$ is the largest power of $p$ dividing $|Z(H)|$.

\smallskip
\noindent
Proof: Since $\Gamma_H \cong \Gamma_G$ we get
\[ |N| - |Z(H)| = |M| - |Z(G)| \] 
\[|Z(H)|([N \colon Z(H)] - 1) = p^ra(p^t - 1)\]
Since $p$ divides $[N \colon Z(H)]$ by Claim 1 and $(p,a)=1$ by assumption, we deduce that $p^r$ divides $|Z(H)|$ and it is the largest power of $p$ dividing it.

\bigskip
\item[Claim 3] $p = q$.

\smallskip
\noindent
Proof: Let $Q \in \Syl_q(Z(H))$ (possibly $|Q|=1$). 
Since $\Gamma_H \cong \Gamma_G$ we get
\[ |N| - |C_H(h)| = |M| - |C_G(g)|\]
\[|Q|([N \colon Q] - [Z(H) \colon Q]q) = p^{r+t}a - p^{r+u}a\]
Note that $p^{r+1}$ divides $p^{r+t}a - p^{r+u}a$. If $p\neq q$, then  $p^{r+1}$ must  divide $([N \colon Q] - [Z(H) \colon Q]q)$. By Claims $1$ and $2$, we have that $p^{r+1}$ divides $[N \colon Z(H)][Z(H) \colon Q] = [N\colon Q]$. Hence we deduce that $p^{r+1}$ divides $[Z(H) \colon Q]q$ and so it must divide $|Z(H)|$, contradicting Claim $2$. Thus we deduce that $p=q$. 

\bigskip
\item[Claim 4] $p^t$ divides $[N \colon Z(H)]$.

\smallskip
\noindent
Proof: Since $\Gamma_H \cong \Gamma_G$ and $p=q$ by Claim $3$, we get
\[ |H| - |N| = |G| - |M|\]
\[[N \colon Z(H)]|Z(H)|(p - 1) = p^rp^ta ([G \colon M] - 1)\]
By Claim $2$ we deduce that $p^t$ divides $[N \colon Z(H)]$. 
\end{description}

We are now ready to reach a contradiction.
Since $p^t$ divides $[N \colon Z(H)]$ and $|C_H(h)|=p|Z(H)|$, we deduce that $|N| \geq |C_H(h)|$. 
 Recall that 
\begin{equation}\label{type1eq} |N| - |C_H(h)| = |M| - |C_G(g)|.\end{equation}

In particular  $|M| \geq |C_G(g)|$, so $p^t \geq p^u$ and Claim $4$ implies that $p^u$ divides $[N \colon Z(H)]$. Recalling that $p=q$, from Equation (\ref{type1eq}) we also get
\[ |Z(H)|([N \colon Z(H)] - p) = ap^rp^u(p^{t-u} - 1).\]
Since $p^r$ is the largest power of $p$ dividing $|Z(H)|$, we conclude that $p^u$ divides $([N \colon Z(H)] - p)$ and so $p^u$ divides $p$. Hence $u=1$, that is, $|C_G(g)| = p|Z(G)|$. Now
\[ |C_H(h)| - |Z(H)| = |C_G(g)| - |Z(G)| \] 
\[ |Z(H)|(p-1) = |Z(G)|(p-1) \]
implying $|Z(H)| = |Z(G)|$ and so $|H| = |G|$.
Therefore by Theorem \ref{same.order}  the group $H$ is nilpotent, a contradiction.

\end{proof}

\begin{lemma}\label{type2.4}
$H$ is not of type (2) or (4).
\end{lemma}

\begin{proof}
Aiming for a contradiction, suppose $H$ is of type (2) or (4). Then by \cite[Lemma 3.11 and 3.12]{AAM} we have $|G|=|H|$ and so by Theorem \ref{same.order} we deduce that $H$ is nilpotent, a contradiction.
\end{proof}

Note that Lemmas \ref{type1} and \ref{type2.4} imply that either $H$ is nilpotent or it is as described in part (3) of Theorem \ref{classification}.
In order to prove Theorem \ref{main} we must show that in the latter case we should have $|Z(G)| < |Z(H)|$, with some extra properties. 

\begin{theorem}\label{exception} 
Suppose $H$ is of type (3) with $[F \colon Z(H)] = q^f$, where $F/Z(H)$ is the kernel of $H/Z(H)$, $q$ is a prime and $f\geq 1$ is an integer. Then
\begin{enumerate}
\item $p \neq q$, that is, the Sylow $q$-subgroups of $G$ are abelian (possibly trivial);
\item $|Z(H)| > |Z(G)|$;
\item $[P \colon Z(P)] > p^4$;
\item none of the maximal subgroups of $G$ is abelian (in particular $P$ does not have maximal nilpotency class). 
\end{enumerate}
\end{theorem}

\begin{proof}
Set $|Z(H)| = bq^u$ for some integers $b, u \geq 1$ with $(b,q)=1$, so $|F| = bq^{u+f}$.  
Let $K/Z(H)$ denote a Frobenius complement of the group $H/Z(H)$. Then 
$|K| = [K \colon Z(H)]|Z(H)| = cbq^u$, for some integer $c\geq 2$ such that  $c$ divides $q^f -1$ (and so in particular $(c,q) = 1$). 
Take $k \in K \backslash Z(H)$, so $K = C_H(k)$, and $x_1, \dots, x_v \in F \backslash Z(H)$, one for each size of $H$-conjugacy classes, so $C_H(x_i) = C_F(x_i)$. In particular 
\[|C_H(k)| = |K| = c|Z(H)| \quad \text{ and } \quad |C_H(x_i)| = q^{t_i}|Z(H)|\] for distinct integers $t_i\geq 1$. Note that from $(c,q)=1$ we see that $|C_H(k)| \neq |C_H(x_i)|$.

Aiming for a contradiction, suppose there exists an isomorphism of graphs
$\phi \colon \Gamma_H \rightarrow \Gamma_G$. Let $g= \Phi(k)$, $M=C_G(g)$ and $y_i=\phi(x_i)$ for every $1 \leq i \leq v$. Note that $G/Z(G)$ is a $p$-group, so we can set $[M \colon Z(G)] = p^m$  for some integer $m\geq 1$ and $[C_G(y_i) \colon Z(G)]=p^{s_i}$ for some distinct integers $s_i\geq 1$.
Also, $|K| - |C_H(x_i)| = |M| - |C_G(y_i)|$, so $|M| \neq |C_G(y_i)|$, that is, $m \neq s_i$ for every $i\geq 1$.

\begin{description}
\item[Claim 1] $p \neq q$, so part (1) holds.

\smallskip
\noindent
Proof: 
Since $\Gamma_H \cong \Gamma_G$ we have
\begin{equation}\label{eq2}|H| - |K| = |G| - |M| \Longrightarrow
q^ubc(q^f- 1) = p^{m+r}a(p^{n-m}- 1)\end{equation}

\vspace{0.2cm}
and
\begin{equation}\label{eq3} |C_H(x_i)| - |Z(H)| = |C_G(y_i)| - |Z(G)| \Longrightarrow
q^ub(q^{t_i} - 1) = p^{r}a(p^{s_i}- 1) .\end{equation}

Aiming for a contradiction, suppose $p=q$. Then $(p,b)=(p,c)=(p,a)=1$. From Equation (\ref{eq2}) we deduce that $p^ u = p^{m+r}$ and  from Equation (\ref{eq3}) we get $p^ u = p^{r}$, a contradiction. Therefore $p\neq q$.

\bigskip
\item[Claim 2] for every $d\leq r$ we have that $p^d$ divides $bc$ if and only if $p^d$ divides $b$.

\smallskip
\noindent
Proof: 
By assumption $p^d$ divides $p^r$ and by Claim $1$ we have $p\neq q$. We conclude by the following equality:
\begin{equation}\label{eq4} |K| - |Z(H)| = |M| - |Z(G)| 
\Longrightarrow  q^u( bc - b ) = p^ra(p^m - 1)\end{equation}

\bigskip
\item[Claim 3] $p^{r+1}$ does not divide $bc$ and $q^{u+1}$ does not divide $a$. In particular $(p,c)=1$ and so $(p,[H \colon Z(H)])=1$.

\smallskip
\noindent
Proof:  Recall that $|K| - |C_H(x_i)| \neq 0 $ since $c$ does not divide $q$,  and so  $|M| \neq |C_g(y)|$. Note that
\begin{equation}\label{eq5} |K| - |C_H(x_i)| = |M| - |C_G(y_i)| \Longrightarrow 
q^u(bc - bq^{t_i}) = p^ra(p^m - p^{s_i}). \end{equation}

From Claim 1 we know $p\neq q$ and by Equation (\ref{eq5}) we get that $p^{r+1}$ divides $(bc - bq^{t_i})$. 
Aiming for a contradiction, suppose $p^{r+1}$ divides $bc$. Then $p^{r+1}$ must divide $b$ and by Equation (\ref{eq4}) we deduce that $p^{r+1}$ divides $p^ra(p^m-1)$, a contradiction. Therefore $p^{r+1}$ does not divide $bc$.

Similarly, if $q^{u+1}$ divides $a$,  then Equation (\ref{eq5}) implies that $q$ divides $bc$, reaching again a contradiction.

Finally, if $p$ divides $c$, then $p$ does not divide $(c-1)$ and by Equation (\ref{eq4}) we deduce that $p^r$ divides $b$. But then $p^{r+1}$ divides $bc$, a contradiction.

\bigskip
\item[Claim 4] We have
\begin{equation}\label{eqC} p^{m}(p^{n-m}- 1)(c - 1)= (p^{m}- 1)c(q^f - 1).\end{equation} In particular 
\begin{enumerate}
\item $n > 2m$;
\item $c < p^m < q^f < p^{n-m}$; and
\item $|Z(H)| > |Z(G)|$, so part (2) holds.
\end{enumerate}

\smallskip
 \noindent
Proof: From Equation (\ref{eq2}) we get
\[ |Z(H)| = q^ub = \frac{p^{r+m}a(p^{n-m}- 1)}{c(q^f - 1)}.\]
Substituting this into Equation (\ref{eq4}) we obtain
$$ p^{m}(p^{n-m}- 1)(c - 1)= (p^{m}- 1)c(q^f - 1),$$
as wanted.

Recall that $c$ divides $q^f-1$, so $c \leq q^f-1$. Aiming for a contradiction, suppose $c=q^f-1$. Then Equation (\ref{eqC}) becomes
\[ p^{m}(p^{n-m}- 1)(c - 1)= (p^{m}- 1)c^2 , \]
that we can rewrite as
\[(p^n- p^m(1+c) +c)c = p^{m}(p^{n-m}- 1). \]

In particular $p^m$ divides $(p^n- p^m(1+c) +c)c $. However, $p^m$ divides $p^n$ while $(p,c)=1$ by Claim 3, and we reach a contradiction. Thus $c<q^f-1$.

As a consequence, we have $c-1 \leq q^f-1$ and from Equation (\ref{eqC}) we deduce that we must have $p^{m}(p^{n-m}- 1) \geq (p^{m}- 1)c$ and so 
\[c \leq \frac{p^{m}(p^{n-m}- 1)}{ (p^{m}- 1)}. \]

Rewrite Equation (\ref{eqC}) as
$$(p^{n-m}- 1)(p^mc - p^m)= (p^{m}c- c)(q^f - 1)$$
and notice that
$$c > p^m \Leftrightarrow p^mc - p^m > p^mc - c \Leftrightarrow p^{n-m}-1 < q^f-1
\Leftrightarrow p^{n-m} < q^f.$$  

Aiming for a contradiction suppose $n \leq 2m$. Then $p^{n-m} \leq p^m$ and by Equation (\ref{eqC}) $p^m$ divides $q^f-1$, so $p^{n-m} \leq p^m \leq q^f-1$.
Hence by what we noted above we obtain $c \geq p^m$. On the other hand, $c \leq \frac{p^{m}(p^{n-m}- 1)}{ (p^{m}- 1)} \leq p^m$. Thus the only option is $c=p^m$, contradicting Claim 3. Therefore $n > 2m$.

Now, if $c > p^m$ then $q^f > p^{n-m} > p^m$ contradicting the fact that by Equation (\ref{eqC}) $p^m$ divides $q^f-1$. Therefore $c < p^m < q^f < p^{n-m}$. Finally, from Equation (\ref{eq4}) we conclude that $|Z(H)| > |Z(G)|$.

\bigskip
\item[Claim 5] for every $i\geq 1$ we have $n > 2s_i$.

\smallskip
\noindent
Proof: 
Note that, using once again the fact that $\Gamma_H \cong \Gamma_G$, we have

\begin{equation}\label{eq6} |H| - |C_H(x_i)| = |G| - |C_G(y_i)| \Longrightarrow 
q^{u+t_i}b(q^{f-t_i}c - 1) = p^{r+s_i}a(p^{n-s_i} - 1). \end{equation}

From Equation (\ref{eq6}) we get
\[ |Z(H)| = q^ub = \frac{p^{r+s_i}a(p^{n-s_i}- 1)}{q^{t_i}(q^{f-t_i}c - 1)}.\]
Substituting this into Equation (\ref{eq1}) we obtain
\[ \frac{p^{r+s_i}a(p^{n-s_i}- 1)}{q^{t_i}(q^{f-t_i}c - 1)}\cdot (q^fc - 1)= p^{r}a(p^{n}- 1) \]
and so 
\[ p^{s_i}(p^{n-s_i}- 1)(q^fc - 1) = (p^{n}- 1)q^{t_i}(q^{f-t_i}c - 1). \]
By Claim 1, $p\neq q$. So we deduce that $p^{s_i}$ divides $q^{f-t_i}c - 1$. In particular, 
\begin{equation}\label{last}
p^{s_i} \leq q^{f-t_i}c - 1 \leq q^{f-t_i}c.
\end{equation}

Since $\Gamma_H \cong \Gamma_G$, we also obtain
\begin{equation}\label{eq1} |H| - |Z(H)| = |G| - |Z(G)| \Longrightarrow 
q^{u}b(q^fc - 1) = p^{r}a(p^{n} - 1). \end{equation}
From Equation (\ref{eq3}) we get
\[ |Z(H)| = q^ub = \frac{p^{r}a(p^{s_i}- 1)}{q^{t_i} - 1}.\]
Substituting this into Equation (\ref{eq1}) we obtain
\[ q^fc = \frac{(p^{n}- 1)(q^{t_i} - 1)}{p^{s_i}- 1} + 1. \]

Now, $p^n - 1 = (p^s_i - 1)p^{n-s_i} +(p^{n-s_i} - 1)$ and so we deduce  

\[\begin{aligned} q^fc &= \frac{((p^s_i - 1)p^{n-s_i} +(p^{n-s_i} - 1))(q^{t_i} - 1)}{p^{s_i}- 1} + 1 
\\ &= p^{n-s_i}q^{t_i}  + \frac{(p^{n-s_i} - 1)(q^{t_i} - 1)}{p^{s_i}- 1} + 1 - p^{n-s_i}
\\ &= p^{n-s_i}q^{t_i} - \frac{(p^{s_i}-q^{t_i})(p^{n-s_i}-1)}{p^{s_i} -1}.
\end{aligned} \]


By Claim 4 we have $|Z(H)| > |Z(G)|$, so by Equation (\ref{eq3}) we conclude that $p^{s_i} > q^{t_i}$. Thus $\frac{(p^{s_i}-q^{t_i})(p^{n-s_i}-1)}{p^{s_i} -1} > 0$ and $q^fc < p^{n-s_i}q^{t_i}$. Combining this with (\ref{last}) we get
\[ p^{s_i}q^{t_i} \leq  q^fc < p^{n-s_i}q^{t_i}.\]
Therefore $p^{s_i} < p^{n-s_i}$, implying $n > 2s_i$.
\end{description}

Now, to prove parts (3) and (4) of the statement, recall that for every $i\geq 1$ we have $m \neq s_i$ and Claims 4 and 5 give $n > 2m$ and $n > 2s_i$. Thus we conclude that $n >4$ and $s_i \neq n-1 \neq m$, implying that the maximal subgroups of $G$ are not abelian. In particular, by Lemma \ref{lem:p-groups}(2) we conclude that if $P$ has nilpotency class greater than $2$ then $C_P(Z_2(P))$ is not a maximal subgroup of $P$. Hence $P$ does not have maximal nilpotency class.



\end{proof}

\begin{proof}[Proof of Theorem \ref{main}] 
By Lemma \ref{solvable} the group $H$ is a finite solvable AC-group and so it is of one of the types described in Theorem \ref{classification}.
Theorem \ref{main} is now a direct consequence of Lemmas \ref{type1} and \ref{type2.4} and of Theorem \ref{exception}.
\end{proof}

\section{Acknowledgements}
The authors are members of the ``National Group for Algebraic and Geometric Structures, and their
Applications'' (GNSAGA - INdAM).

\bibliographystyle{amsplain}
\bibliography{books}
\end{document}